\documentclass[11pt,reqno]{amsart}
\usepackage{amsmath,amssymb,latexsym,esint,cite,mathrsfs}
\usepackage{verbatim,wasysym,cite,relsize,color}
\usepackage[left=3cm,right=3cm,top=3cm,bottom=3cm]{geometry}
\usepackage[latin1]{inputenc}
\usepackage{microtype}
\usepackage{color,enumitem,graphicx}
\usepackage[colorlinks=true,urlcolor=blue, citecolor=red,linkcolor=blue,
linktocpage,pdfpagelabels, bookmarksnumbered,bookmarksopen]{hyperref}
\usepackage[hyperpageref]{backref}
\usepackage[english]{babel}
\usepackage{physics} 

\numberwithin{equation}{section}
\newtheorem{theorem}{Theorem}[section]
\newtheorem{lemma}[theorem]{Lemma}

\newtheorem{definition}[theorem]{Definition}
\newtheorem{proposition}[theorem]{Proposition}

\renewcommand{\epsilon}{{\varepsilon}}

\title[Instability standing waves system NLS quadratic interaction]
{Strong instability of standing waves for a system NLS with quadratic interaction}
\author[V. D. Dinh]{Van Duong Dinh}
\address[V. D. Dinh]{Institut de Math\'ematiques de Toulouse UMR5219, Universit\'e Toulouse CNRS, 31062 Toulouse Cedex 9, France 
and 
Department of Mathematics, HCMC University of Pedagogy, 280 An Duong Vuong, Ho Chi Minh, Vietnam}
\email{dinhvan.duong@math.univ-toulouse.fr}


\begin{document}
	
	\begin{abstract}
	We study the strong instability of standing waves for a system of nonlinear Schr\"odinger equations with quadratic interaction under the mass resonance condition in dimension $d=5$. 
	\end{abstract}

	\maketitle

	\section{Introduction}
	\label{S:0}
	We consider the system NLS equations
	\begin{equation} \label{Syst-NLS}
		\left\{ 
		\renewcommand*{\arraystretch}{1.3}
		\begin{array}{rcl}
			i\partial_t u + \frac{1}{2m} \Delta u & = & \lambda v \overline{u}, \\
			i\partial_t v + \frac{1}{2M} \Delta v & = & \mu u^2,
		\end{array}
		\right.
	\end{equation}
	where $u, v: \mathbb{R} \times \mathbb{R}^d \rightarrow \mathbb{C}$, $m$ and $M$ are positive constants, $\Delta$ is the Laplacian in $\mathbb{R}^d$ and $\lambda, \mu$ are complex constants. 
	
	The system \eqref{Syst-NLS} is regarded as a non-relativistic limit of the system of nonlinear Klein-Gordon equations
	\[
	\left\{ 
	\renewcommand*{\arraystretch}{1.3}
	\begin{array}{rcl}
	\frac{1}{2c^2m}\partial^2_t u - \frac{1}{2m} \Delta u + \frac{mc^2}{2} u& = & -\lambda v \overline{u}, \\
	\frac{1}{2c^2M}\partial^2_t v - \frac{1}{2M} \Delta v + \frac{Mc^2}{2} v& = & -\mu u^2,
	\end{array}
	\right.
	\]
	under the mass resonance condition
	\begin{align} \label{mas-res}
		M=2m.
	\end{align}
	Indeed, the modulated wave functions $(u_c,v_c):= (e^{itmc^2} u, e^{itMc^2} v)$ satisfy
	\begin{align}\label{klei-gord}
		\left\{ 
		\renewcommand*{\arraystretch}{1.3}
		\begin{array}{rcl}
			\frac{1}{2c^2m} \partial^2_t u_c - i\partial_t u_c - \frac{1}{2m} \Delta u_c &=& - e^{itc^2(2m-M)} \lambda v_c \overline{u}_c,\\
			\frac{1}{2c^2M} \partial^2_t v_c - i\partial_t v_c - \frac{1}{2M} \Delta v_c &=& - e^{itc^2(M-2m)} \mu u^2_c.
		\end{array}
		\right.
	\end{align}
	We see that the phase oscillations on the right hand sides vanish if and only if \eqref{mas-res} holds, and the system \eqref{klei-gord} formally yields \eqref{Syst-NLS} as the speed of light $c$ tends to infinity. The system \eqref{Syst-NLS} also appears in the interaction process for waves propagation in quadratic media (see e.g. \cite{CMS}).
	
	The system \eqref{Syst-NLS} has attracted a lot of interest in past several years. The scattering theory and the asymptotic behavior of solutions have been studied in  \cite{HLN, HLN-modi, HLO, OU}. The Cauchy problem for \eqref{Syst-NLS} in $L^2$, $H^1$ and in the weighted $L^2$ space $\langle x \rangle^{-1} L^2 = \mathcal{F}(H^1)$ under mass resonance condition have been studied in \cite{HOT}. The space-time analytic smoothing effect has been studied in \cite{HO-1, HO-2, Hoshino}. The sharp threshold for scattering and blow-up for \eqref{Syst-NLS} under the mass resonance condition in dimension $d=5$ has been studied in \cite{Hamano}. The existence, stability of standing waves and the characterization of finite time blow-up solutions with minimal mass have been studied recently in \cite{Dinh}.
	
	Let us recall the local well-posedness in $H^1$ for \eqref{Syst-NLS} due to \cite{HOT}. To ensure the conservation law of total charge, it is natural to consider the following condition:
	\begin{align} \label{mas-con}
		\exists ~ c \in \mathbb{R} \backslash \{0\} \ : \ \lambda = c \overline{\mu}. 
	\end{align} 
	\begin{proposition}[LWP in $H^1$ \cite{HOT}] \label{prop-lwp-h1}
		Let $d\leq 6$ and let $\lambda$ and $\mu$ satisfy \eqref{mas-con}. Then for any $(u_0,v_0) \in H^1\times H^1$, there exists a unique paire of local solutions $(u,v) \in Y(I)\times Y(I)$ of \eqref{Syst-NLS} with initial data $(u(0), v(0))=(u_0,v_0)$, where
		\begin{align*}
			Y(I) = (C\cap L^\infty)(I,H^1) \cap L^4(I,W^{1,\infty}) &\text{ for } d=1, \\
			Y(I) = (C\cap L^\infty)(I,H^1) \cap L^{q_0}(I,W^{1,{r_0}}) &\text{ for } d=2, 
		\end{align*}
		where $0<\frac{2}{q_0}=1-\frac{2}{r_0}<1$ with $r_0$ sufficiently large, 
		\begin{align*}
			Y(I) = (C\cap L^\infty)(I, H^1) \cap L^2(I, W^{1,\frac{2d}{d-2}}) &\text{ for } d\geq 3.
		\end{align*}
		Moreover, the solution satisfies the conservation of mass and energy: for all $t\in I$,
		\begin{align*}
		M(u(t),v(t))&:= \|u(t)\|^2_{L^2}+ c\|v(t)\|^2_{L^2} = M(u_0,v_0), \\
		E(u(t),v(t))&:= \frac{1}{2m}\|\nabla u(t)\|^2_{L^2} + \frac{c}{4M} \|\nabla v(t)\|^2_{L^2} + \emph{Re} (\lambda \langle v(t), u^2(t) \rangle ) = E(u_0,v_0),
		\end{align*}
		where $\langle \cdot, \cdot \rangle$ is the scalar product in $L^2$. 
	\end{proposition}

	We now assume that $\lambda$ and $\mu$ satisfy \eqref{mas-con} with $c>0$ and $\lambda \ne 0, \mu \ne 0$. By change of variables
	\[
	u(t,x) \mapsto \sqrt{\frac{c}{2}} |\mu| u \left(t,\sqrt{\frac{1}{2m}} x \right), \quad v(t,x) \mapsto -\frac{\lambda}{2} v\left( t, \sqrt{\frac{1}{2m}} x\right),
	\]
	the system \eqref{Syst-NLS} becomes 
	\begin{equation} \label{cha-Syst}
		\left\{ 
		\renewcommand*{\arraystretch}{1.3}
		\begin{array}{rcl}
			i\partial_t u + \Delta u & = & -2 v \overline{u}, \\
			i\partial_t v + \kappa \Delta v & = & - u^2,
		\end{array}
		\right.
	\end{equation}
	where $\kappa=\frac{m}{M}$ is the mass ratio. Note that the mass and the energy now become
	\begin{align*}
	M(u(t),v(t)) &= \|u(t)\|^2_{L^2} + 2 \|v(t)\|^2_{L^2}, \\
	E(u(t),v(t)) &= \frac{1}{2} (\|\nabla u(t)\|^2_{L^2} + \kappa \|\nabla v(t)\|^2_{L^2} ) - \text{Re}( \langle v(t), u^2(t)\rangle).
	\end{align*}
	The local well-posedness in $H^1$ for \eqref{cha-Syst} reads as follows.
	\begin{proposition} [LWP in $H^1$] \label{prop-lwp-wor}
		Let $d\leq 6$. Then for any $(u_0, v_0) \in H^1 \times H^1$, there exists a unique pair of local solutions $(u,v) \in Y(I) \times Y(I)$ of \eqref{cha-Syst} with initial data $(u(0), v(0))=(u_0,v_0)$. Moreover, the solution satisfies the conservation of mass and energy: for all $t \in I$,
		\begin{align*}
		M(u(t),v(t)) &:= \|u(t)\|^2_{L^2} + 2 \|v(t)\|^2_{L^2} = M(u_0,v_0), \\
		E(u(t),v(t)) &:= \frac{1}{2} (\|\nabla u(t)\|^2_{L^2} + \kappa \|\nabla v(t)\|^2_{L^2}) - \emph{Re} (\langle v(t),u^2(t)\rangle) = E(u_0,v_0).
		\end{align*}
	\end{proposition}
	
	The main purpose of this paper is to study the strong instability of standing waves for the system \eqref{cha-Syst} under the mass resonance condition $\kappa=\frac{1}{2}$ in dimension $d=5$. Let $d=5$ and consider
	\begin{equation} \label{mas-res-Syst}
	\left\{ 
	\renewcommand*{\arraystretch}{1.3}
	\begin{array}{rcl}
	i\partial_t u + \Delta u & = & -2 v \overline{u}, \\
	i\partial_t v + \frac{1}{2} \Delta v & = & - u^2,
	\end{array}
	\right.
	\end{equation}
	We call a standing wave a solution to \eqref{mas-res-Syst} of the form $(e^{i\omega t} \phi_\omega, e^{i2\omega t} \psi_\omega)$, where $\omega \in \mathbb{R}$ is a frequency and $(\phi_\omega, \psi_\omega) \in H^1 \times H^1$ is a non-trivial solution to the elliptic system
	\begin{equation} \label{ell-equ}
	\left\{ 
	\renewcommand*{\arraystretch}{1.3}
	\begin{array}{rcl} -\Delta \phi_\omega + \omega \phi_\omega & = & 2 \psi_\omega \overline{\phi}_\omega, \\ -\frac{1}{2} \Delta \psi_\omega + 2\omega \psi_\omega & = & \phi_\omega^2.\end{array}
	\right.
	\end{equation}
	We are interested in showing the strong instability of ground state standing waves for \eqref{mas-res-Syst}. Let us first introduce the notion of ground states related to \eqref{mas-res-Syst}. Denote
	\[
	S_\omega(u,v):= E(u,v) + \frac{\omega}{2} M(u,v) = \frac{1}{2} K(u,v) + \frac{\omega}{2} M(u,v) - P(u,v),
	\]
	where
	\[
	K(u,v) = \|\nabla u\|^2_{L^2} + \frac{1}{2} \|\nabla v\|^2_{L^2}, \quad M(u,v) = \|u\|^2_{L^2} + 2 \|v\|^2_{L^2}, \quad P(u,v) = \text{Re} \int \overline{v} u^2 dx.
	\]
	We also denote the set of non-trivial solutions of \eqref{ell-equ} by
	\[
	\mathcal{A}_\omega:= \{ (u,v) \in H^1 \times H^1 \backslash \{(0,0)\} \ : \ S'_\omega(u,v) =0 \}.
	\]
	\begin{definition} \label{def-gro-sta-ins}
		A pair of functions $(\phi,\psi) \in H^1 \times H^1$ is called ground state for \eqref{ell-equ} if it is a minimizer of $S_\omega$ over the set $\mathcal{A}_\omega$. The set of ground states is denoted by $\mathcal{G}_\omega$. In particular,
		\[
		\mathcal{G}_\omega= \{(\phi,\psi) \in \mathcal{A}_\omega \ : \ S_\omega(\phi,\psi) \leq S_\omega(u,v), \forall (u,v) \in \mathcal{A}_\omega \}.
		\]
	\end{definition}
	We have the following result on the existence of ground states for \eqref{ell-equ}.
	\begin{proposition} \label{prop-exi-gro-sta-ins}
		Let $d=5$, $\kappa = \frac{1}{2}$ and $\omega>0$. Then the set $\mathcal{G}_\omega$ is not empty, and it is characterized by
		\[
		\mathcal{G}_\omega = \{ (u,v) \in H^1 \times H^1 \backslash \{(0,0)\} \ : \ S_\omega(u,v) = d(\omega), K_\omega(u,v) =0 \},
		\]
		where 
		\[
		K_\omega(u,v) = \left. \partial_\gamma S_\omega(\gamma u, \gamma v) \right|_{\gamma=1} = K(u,v) + \omega M(u,v) - 3 P(u,v)
		\]
		is the Nehari functional and
		\begin{align} \label{d-ome}
		d(\omega) := \inf \{ S_\omega(u,v) \ : \ (u,v) \in H^1 \times H^1 \backslash \{(0,0)\}, K_\omega(u,v) =0 \}. 
		\end{align}
	\end{proposition}
	The existence of real-valued ground states for \eqref{ell-equ} was proved in \cite{HOT} (actually for $d\leq 5$ and $\kappa>0$). Here we proved the existence of ground states (not necessary real-valued) and proved its characterization. This characterization plays an important role in the study of strong instability of ground states standing waves for \eqref{mas-res-Syst}. We only state and prove Proposition $\ref{prop-exi-gro-sta-ins}$ for $d=5$ and $\kappa=\frac{1}{2}$. However, it is still available for $d\leq 5$ and $\kappa>0$.
	
	We also recall the definition of the strong instability of standing waves. 
	\begin{definition} \label{def-str-ins}
		We say that the standing wave $(e^{i\omega t} \phi_\omega, e^{i2\omega t} \psi_\omega)$ is strongly unstable if for any $\epsilon>0$, there exists $(u_0,v_0) \in H^1 \times H^1$ such that $\|(u_0,v_0) - (\phi_\omega, \psi_\omega)\|_{H^1 \times H^1} <\epsilon$ and the corresponding solution $(u(t),v(t))$ to \eqref{mas-res-Syst} with initial data $(u(0), v(0))=(u_0,v_0)$ blows up in finite time.
	\end{definition}
	Our main result of this paper is the following.
	\begin{theorem} \label{theo-str-ins}
		Let $d=5$, $\kappa=\frac{1}{2}$, $\omega>0$ and $(\phi_\omega, \psi_\omega) \in \mathcal{G}_\omega$. Then the ground state standing waves $(e^{i\omega t} \phi_\omega, e^{i2\omega t} \psi_\omega)$ for \eqref{mas-res-Syst} is strongly unstable.
	\end{theorem}
	
	To our knowledge, this paper is the first one addresses the strong instability of standing waves for a system of nonlinear Schr\"odinger equations with quadratic interaction. In \cite{CCO-ins}, Colin-Colin-Ohta proved the instability of standing waves for a system of nonlinear Schr\"odinger equations with three waves interaction. However, they only studied the orbital instability not strong instability by blow-up, and they only consider a special standing wave solution $(0,0,e^{2i\omega t} \varphi)$, where $\varphi$ is the unique positive radial solution to the elliptic equation
	\[
	-\Delta \varphi + 2 \omega \varphi - |\varphi|^{p-1} \varphi=0.
	\]
	
	This paper is organized as follows. In Section $\ref{S:1}$, we show the existence of ground states and its characterization given in Proposition $\ref{prop-exi-gro-sta-ins}$. In Section $\ref{S:2}$, we give the proof of the strong instability of standing waves given in Theorem $\ref{theo-str-ins}$. 
	\section{Exitence of ground states}
	\label{S:1}
	
	We first show the existence of ground states given in Proposition $\ref{prop-exi-gro-sta-ins}$. To do so, we need the following profile decomposition.
	\begin{proposition}[Profile decomposition] \label{prop-pro-dec-5D}
		Let $d=5$ and $\kappa=\frac{1}{2}$. Le $(u_n,v_n)_{n\geq 1}$ be a bounded sequence in $H^1 \times H^1$. Then there exist a subsequence, still denoted by $(u_n,v_n)_{n\geq 1}$, a family $(x^j_n)_{n\geq 1}$ of sequences in $\mathbb{R}^5$ and a sequence $(U^j, V^j)_{j\geq 1}$ of $H^1\times H^1$-functions such that
		\begin{itemize}
			\item[(1)] for every $j\ne k$, 
			\begin{align} \label{ort-pro-dec-5D}
			|x^j_n-x^k_n| \rightarrow \infty \text{ as } n\rightarrow \infty; 
			\end{align}
			\item[(2)] for every $l\geq 1$ and every $x \in \mathbb{R}^5$, 
			\[
			u_n(x) = \sum_{j=1}^l U^j(x-x^j_n) + u^l_n(x), \quad v_n(x)= \sum_{j=1}^l V^j(x-x^j_n) + v^l_n(x),
			\]
			with
			\begin{align} \label{err-pro-dec-5D}
			\limsup_{n\rightarrow \infty} \|(u^l_n, v^l_n)\|_{L^q\times L^q} \rightarrow 0 \text{ as } l \rightarrow \infty,
			\end{align}
			for every $q\in (2, 10/3)$.
		\end{itemize}
		Moreover, for every $l\geq 1$,
		\begin{align}
		M(u_n,v_n) &= \sum_{j=1}^l M(U^j_n, V^j_n) + M(u^l_n,v^l_n) + o_n(1), \label{mas-pro-dec-5D} \\
		K(u_n,v_n) &= \sum_{j=1}^l K(U^j,V^j) + K(u^l_n, v^l_n) + o_n(1), \label{kin-pro-dec-5D} \\
		P(u_n,v_n) &= \sum_{j=1}^l P(U^j,V^j) + P(u^l_n, v^l_n) + o_n(1), \label{sca-pro-dec-5D}
		\end{align}
		where $o_n(1) \rightarrow 0$ as $n\rightarrow \infty$.
	\end{proposition}
	We refer the reader to \cite[Proposition 3.5]{Dinh} for the proof of this profile decomposition. 
	
	The proof of Proposition $\ref{prop-exi-gro-sta-ins}$ is done by several lemmas. To simplify the notation, we denote for $\omega>0$,
	\[
	H_\omega(u,v):= K(u,v) + \omega M(u,v).
	\]
	It is easy to see that for $\omega>0$ fixed, 
	\begin{align} \label{equ-nor}
	H_\omega(u,v) \sim \|(u,v)\|_{H^1 \times H^1}.
	\end{align}
	Note also that
	\[
	S_\omega(u,v)=\frac{1}{2} K_\omega(u,v)+\frac{1}{2}P(u,v)= \frac{1}{3}K_\omega(u,v)+\frac{1}{6} H_\omega(u,v).
	\]
	
	\begin{lemma} \label{lem-pos-d-ome}
		$d(\omega)>0$.
	\end{lemma}
	
	\begin{proof}
		Let $(u,v) \in H^1 \times H^1 \backslash \{(0,0)\}$ be such that $K_\omega(u,v)=0$ or $H(u,v) = 3 P(u,v)$. We have from Sobolev embedding that
		\begin{align*}
		P(u,v) \leq \int |v| |u|^2 dx \lesssim \|v\|_{L^3} \|u\|^2_{L^3}  \lesssim \|\nabla v\|_{L^2} \|\nabla u\|^2_{L^2} \lesssim [H_\omega(u,v)]^3 \lesssim [P(u,v)]^3.
		\end{align*}
		This implies that there exists $C>0$ such that
		\[
		P(u,v) \geq \sqrt{\frac{1}{C}}>0.
		\]
		Thus
		\[
		S_\omega(u,v) = \frac{1}{2} K(u,v) + \frac{1}{2} P(u,v) \geq \frac{1}{2} \sqrt{\frac{1}{C}}>0.
		\]
		Taking the infimum over all $(u,v)\in H^1 \times H^1 \backslash \{(0,0)\}$ satisfying $K_\omega(u,v)=0$, we get the result.
	\end{proof}
	
	We now denote the set of all minimizers for $d(\omega)$ by
	\[
	\mathcal{M}_\omega := \left\{ (u,v) \in H^1 \times H^1 \backslash \{(0,0)\} \ : \ K_\omega(u,v) =0, S_\omega(u,v) = d(\omega) \right\}.
	\]
	\begin{lemma} \label{lem-no-emp-M-ome}
		The set $\mathcal{M}_\omega$ is not empty.
	\end{lemma}
	
	\begin{proof}
		Let $(u_n,v_n)_{n\geq 1}$ be a minimizing sequence for $d(\omega)$, i.e. $(u_n, v_n) \in H^1 \times H^1 \backslash \{(0,0)\}$, $K_\omega(u_n,v_n) =0$ for any $n\geq 1$ and $\lim_{n\rightarrow \infty} S_\omega(u_n,v_n) = d(\omega)$. Since $K_\omega(u_n,v_n) = 0$, we have that $H_\omega(u_n,v_n) = 3 P(u_n,v_n)$ for any $n\geq 1$. We also have that 
		\[
		S_\omega (u_n, v_n) = \frac{1}{3} K_\omega(u_n,v_n) + \frac{1}{6}H_\omega(u_n,v_n) \rightarrow d(\omega) \text{ as } n\rightarrow \infty.
		\]
		This yields that there exists $C>0$ such that 
		\[
		H_\omega(u_n,v_n) \leq 6 d(\omega) + C
		\]
		for all $n\geq 1$. By \eqref{equ-nor}, $(u_n,v_n)_{n\geq 1}$ is a bounded sequence in $H^1 \times H^1$. We apply the profile decomposition given in Proposition $\ref{prop-pro-dec-5D}$ to get up to a subsequence,
		\[
		u_n(x) = \sum_{j=1}^l U^j(x-x^j_n) + u^l_n(x), \quad v_n(x) = \sum_{j=1}^l V^j(x-x^j_n) + v^l_n(x)
		\]
		for some family of sequences $(x^j_n)_{n\geq 1}$ in $\mathbb{R}^5$ and $(U^j,V^j)_{j\geq 1}$ a sequence of $H^1 \times H^1$-functions satisfying \eqref{err-pro-dec-5D} -- \eqref{sca-pro-dec-5D}. We see that
		\[
		H_\omega (u_n,v_n)=\sum_{j=1}^l H_\omega (U^j,V^j) + H_\omega(u^l_n, v^l_n) + o_n(1).
		\]
		This implies that
		\begin{align*}
		K_\omega(u_n,v_n)&=H_\omega(u_n,v_n)-3P(u_n,v_n)\\
		&=\sum_{j=1}^l H_\omega(U^j,V^j) + H_\omega(u^l_n,v^l_n) - 3P(u_n,v_n)+ o_n(1) \\
		&=\sum_{j=1}^l K_\omega(U^j,V^j) + 3\sum_{j=1}^l P(U^j,V^j)- 3 P(u_n,v_n) + H_\omega(u^l_n,v^l_n)+o_n(1).
		\end{align*}
		Since $K_\omega(u_n,v_n)=0$ for any $n\geq 1$, $P(u_n,v_n) \rightarrow 2 d(\omega)$ as $n\rightarrow \infty$ and $H_\omega(u^l_n,v^l_n) \geq 0$ for any $n\geq 1$, we infer that
		\[
		\sum_{j=1}^l K_\omega(U^j,V^j) + 3 \sum_{j=1}^l P(U^j,V^j) - 6 d(\omega) \leq 0
		\]
		or
		\[
		\sum_{j=1}^l H_\omega (U^j,V^j) - 6d(\omega) \leq 0.
		\]
		By H\"older's inequality and \eqref{err-pro-dec-5D}, it is easy to see that $\limsup_{n\rightarrow \infty} P(u_n^l, v^l_n) =0$ as $l\rightarrow \infty$. Thanks to \eqref{sca-pro-dec-5D}, we have that
		\[
		2d(\omega) = \lim_{n\rightarrow \infty} P(u_n,v_n) = \sum_{j=1}^\infty P(U^j,V^j).
		\]
		We thus obtain
		\begin{align} \label{pro-dec-app-5D}
		\sum_{j=1}^\infty K_\omega(U^j,V^j) \leq 0 \text{ and } \sum_{j=1}^\infty H_\omega(U^j,V^j) \leq 6d(\omega).
		\end{align}
		We now claim that $K_\omega(U^j,V^j) =0$ for all $j\geq 1$. Indeed, suppose that if there exists $j_0 \geq 1$ such that $K_\omega(U^{j_0},V^{j_0}) <0$, then we see that the equation $K_\omega(\gamma U^{j_0}, \gamma V^{j_0}) = \gamma^2 H_\omega(U^{j_0},V^{j_0}) - 3 \gamma^3 P(U^{j_0},V^{j_0})=0$ admits a unique non-zero solution 
		\[
		\gamma_0 := \frac{H_\omega(U^{j_0},V^{j_0})}{3 P(U^{j_0}, V^{j_0})} \in (0,1).
		\]
		By the definition of $d(\omega)$, we have
		\[
		d(\omega) \leq S_\omega(\gamma_0 U^{j_0}, \gamma_0 V^{j_0}) = \frac{1}{6} H_\omega(\gamma_0 U^{j_0}, \gamma_0 V^{j_0}) = \frac{\gamma_0^2}{6} H(U^{j_0},V^{j_0}) <\frac{1}{6} H_\omega(U^{j_0},V^{j_0})
		\]
		which contradicts to the second inequality in \eqref{pro-dec-app-5D}. We next claim that there exists only one $j$ such that $(U^j,V^j)$ is non-zero. Indeed, if there are $(U^{j_1},V^{j_1})$ and $(U^{j_2},V^{j_2})$ non-zero, then by \eqref{pro-dec-app-5D}, both $H_\omega(U^{j_1},V^{j_1})$ and $H_\omega(U^{j_2},V^{j_2})$ are strictly smaller than $6d(\omega)$. Moreover, since $K_\omega(U^{j_1},V^{j_1}) =0$,
		\[
		d(\omega) \leq S_\omega(U^{j_1},V^{j_1}) = \frac{1}{6} H_\omega(U^{j_1},V^{j_1}) <d(\omega)
		\]
		which is absurd. Therefore, without loss of generality we may assume that the only one non-zero profile is $(U^1,V^1)$. We will show that $(U^1,V^1) \in \mathcal{M}_\omega$. Indeed, we have $P(U^1,V^1) = 2d(\omega)>0$ which implies $(U^1,V^1) \ne (0,0)$. We also have 
		\[
		K_\omega(U^1,V^1) =0 \text{ and } S_\omega(U^1,V^1) = \frac{1}{2} P(U^1,V^1) =d(\omega).
		\]
		This shows that $(U^1,V^1)$ is a minimizer for $d(\omega)$. The proof is complete.
	\end{proof}
	
	\begin{lemma} \label{lem-inc-M-G}
		$\mathcal{M}_\omega \subset \mathcal{G}_\omega$.
	\end{lemma}
	
	\begin{proof}
		Let $(\phi,\psi) \in \mathcal{M}_\omega$. Since $K_\omega(\phi,\psi) =0$, we have $H_\omega(\phi,\psi) = 3 P(\phi,\psi)$. On the other hand, since $(\phi,\psi)$ is a minimizer for $d(\omega)$, there exists a Lagrange multiplier $\gamma \in \mathbb{R}$ such that
		\[
		S'_\omega(\phi,\psi) = \gamma K'_\omega(\phi,\psi).
		\]
		This implies that
		\[
		0 = K_\omega(\phi,\psi) = \langle S'_\omega(\phi,\psi), (\phi,\psi)\rangle = \gamma \langle K'_\omega(\phi,\psi), (\phi,\psi)\rangle.
		\]
		A direct computation shows that
		\[
		\langle K'_\omega(\phi,\psi), (\phi,\psi)\rangle = 2 K(\phi,\psi) + 2 \omega M(\phi,\psi) - 9 P(\phi,\psi) = 2 H_\omega(\phi,\psi) - 9 P(\phi,\psi) = - 3P(\phi,\psi) <0.
		\]
		Therefore, $\gamma = 0$ and $S'_\omega(\phi,\psi) =0$ or $(\phi,\psi) \in \mathcal{A}_\omega$. It remains to show that $S_\omega(\phi,\psi) \leq S_\omega(u,v)$ for all $(u,v) \in \mathcal{A}_\omega$. Let $(u,v) \in\mathcal{A}_\omega$. We have $K_\omega(u,v) = \langle S'_\omega(u,v), (u,v) \rangle =0$. By the definition of $d(\omega)$, we get $S_\omega(\phi,\psi) \leq S_\omega(u,v)$. The proof is complete.
	\end{proof}
	
	\begin{lemma} \label{lem-inc-G-M}
		$\mathcal{G}_\omega \subset \mathcal{M}_\omega$.
	\end{lemma}
	
	\begin{proof}
		Let $(\phi_\omega, \psi_\omega) \in \mathcal{G}_\omega$. Since $\mathcal{M}_\omega$ is not empty, we take $(\phi,\psi) \in \mathcal{M}_\omega$. We have from Lemma $\ref{lem-inc-M-G}$ that $(\phi,\psi) \in \mathcal{G}_\omega$. Thus, $S_\omega(\phi_\omega,\psi_\omega) = S_\omega(\phi, \psi)=d(\omega)$. It remains to show that $K_\omega(\phi_\omega,\psi_\omega)=0$. Since $(\phi_\omega,\psi_\omega) \in \mathcal{A}_\omega$, $S'_\omega(\phi_\omega,\psi_\omega)=0$. This implies that
		\[
		K_\omega(\phi_\omega,\psi_\omega) = \langle S'_\omega(\phi_\omega,\psi_\omega), (\phi_\omega,\psi_\omega) \rangle =0.
		\]
		The proof is complete.
	\end{proof}
	
	\noindent \textit{Proof of Proposition $\ref{prop-exi-gro-sta-ins}$.}
	The proof of Proposition $\ref{prop-exi-gro-sta-ins}$ follows immediately from Lemmas $\ref{lem-no-emp-M-ome}$, $\ref{lem-inc-M-G}$ and $\ref{lem-inc-G-M}$.
	\hfill $\Box$
	
	\section{Strong instability of standing waves}
	\label{S:2}
	We are now able to study the strong instability of standing waves for \eqref{mas-res-Syst}. Note that the local well-posedness in $H^1 \times H^1$ for \eqref{mas-res-Syst} in 5D is given in Proposition $\ref{prop-lwp-wor}$. Let us start with the following so-called Pohozaev's identities.
	\begin{lemma} \label{lem-poh-ide}
		Let $d=5$, $\kappa=\frac{1}{2}$ and $\omega>0$. Let $(\phi_\omega,\psi_\omega) \in H^1 \times H^1$ be a solution to \eqref{ell-equ}. Then the following identities hold
		\[
		2 K(\phi_\omega, \psi_\omega) = 5 P(\phi_\omega, \psi_\omega), \quad 2\omega M(\phi_\omega, \psi_\omega) = P(\phi_\omega,\psi_\omega).
		\]
	\end{lemma}
	
	\begin{proof}
		We only make a formal calculation. The rigorous proof follows from a standard approximation argument. Multiplying both sides of the first equation in \eqref{ell-equ} with $\overline{\phi}_\omega$, integrating over $\mathbb{R}^5$ and taking the real part, we have
		\[
		\|\nabla \phi_\omega\|^2_{L^2} + \omega \|\phi_\omega\|^2_{L^2}= 2 \text{Re} \int \overline{\psi}_\omega \phi_\omega^2 dx.
		\]
		Multiplying both sides of the second equation in \eqref{ell-equ} with $\overline{\psi}_\omega$, integrating over $\mathbb{R}^5$ and taking the real part, we get
		\[
		\frac{1}{2} \|\nabla \psi_\omega\|^2_{L^2} + 2 \omega \|\psi_\omega\|^2_{L^2} = \text{Re} \int \overline{\psi}_\omega \phi_\omega^2 dx.
		\]
		We thus obtain
		\begin{align} \label{poh-ide-pro-1}
		K(\phi_\omega,\psi_\omega) + 2 \omega M(\phi_\omega,\psi_\omega) = 3 P(\phi_\omega,\psi_\omega).
		\end{align}
		Multiplying both sides of the first equation in \eqref{ell-equ} with $x \cdot \nabla \overline{\phi}_\omega$, integrating over $\mathbb{R}^5$ and taking the real part, we see that
		\[
		-\text{Re} \int \Delta \phi_\omega x \cdot \nabla \overline{\phi}_\omega dx + \omega \text{Re} \int \phi_\omega x \cdot \nabla \overline{\phi}_\omega dx = 2 \text{Re} \int \psi_\omega \overline{\phi}_\omega x \cdot \nabla \overline{\phi}_\omega dx.
		\]
		A direct computation shows that
		\begin{align*}
		\text{Re} \int \Delta\phi_\omega x \cdot \nabla \overline{\phi}_\omega dx &=\frac{3}{2} \|\nabla \phi_\omega\|^2_{L^2}, \\
		\text{Re} \int \phi_\omega x \cdot \nabla \overline{\phi}_\omega dx &= -\frac{5}{2} \|\phi_\omega\|^2_{L^2}, \\
		\text{Re} \int \psi_\omega \overline{\phi}_\omega x \cdot \nabla \overline{\phi}_\omega dx &= -\frac{5}{2} \text{Re} \int \overline{\psi}_\omega (\phi_\omega)^2 dx - \frac{1}{2} \text{Re} \int \phi_\omega^2  x\cdot \nabla \overline{\psi}_\omega dx.
		\end{align*}
		It follows that
		\[
		-\frac{3}{2} \|\nabla \phi_\omega\|^2_{L^2} - \frac{5}{2} \omega \|\phi_\omega\|^2_{L^2} = - 5 \text{Re} \int \overline{\psi}_\omega \phi^2_\omega dx - \text{Re} \int \phi^2_\omega x \cdot \nabla \overline{\psi}_\omega dx.
		\]
		Similarly, multiplying both sides of the second equation in \eqref{ell-equ} with $x \cdot \nabla \overline{\psi}_\omega$, integrating over $\mathbb{R}^5$ and taking the real part, we have
		\[
		-\frac{3}{4} \|\nabla \psi_\omega\|^2_{L^2} - 5 \omega \|\psi_\omega\|^2_{L^2} = \text{Re} \int \phi^2_\omega x\cdot \nabla \overline{\psi}_\omega dx.
		\]
		We thus get
		\begin{align} \label{poh-ide-pro-2}
		\frac{3}{2} K(\phi_\omega,\psi_\omega) + \frac{5}{2} \omega M(\phi_\omega,\psi_\omega) = 5 P(\phi_\omega,\psi_\omega).
		\end{align}
		Combining \eqref{poh-ide-pro-1} and \eqref{poh-ide-pro-2}, we prove the result.
	\end{proof}
	
	We also have the following exponential decay of solutions to \eqref{ell-equ}.
	\begin{lemma} \label{lem-dec-pro-gro-sta}
		Let $d=5$, $\kappa=\frac{1}{2}$ and $\omega>0$. Let $(\phi_\omega,\psi_\omega) \in H^1 \times H^1$ be a solution to \eqref{ell-equ}. Then the following properties hold
		\begin{itemize}
			\item $(\phi_\omega,\psi_\omega) \in W^{3,p} \times W^{3,p}$ for every $2 \leq p <\infty$. In particular, $(\phi_\omega,\psi_\omega) \in C^2 \times C^2$ and $|D^\beta \phi_\omega(x)| + |D^\beta \psi_\omega (x)| \rightarrow 0$ as $|x| \rightarrow \infty$ for all $|\beta| \leq 2$;
			\item 
			\[
			\int e^{|x|} (|\nabla \phi_\omega|^2 + |\phi_\omega|^2) dx <\infty, \quad \int e^{|x|} (|\nabla \psi_\omega|^2 + 4|\psi_\omega|^2) dx <\infty.
			\]
			In particular, $(|x| \phi_\omega, |x| \psi_\omega) \in L^2 \times L^2$.
		\end{itemize}
	\end{lemma}
	
	\begin{proof}
		The follows the argument of \cite[Theorem 8.1.1]{Cazenave}. Let us prove the first item. We note that if $(\phi_\omega, \psi_\omega) \in L^p \times L^p$ for some $2 \leq p<\infty$, then $\psi_\omega \overline{\phi}_\omega, \phi^2_\omega \in L^{\frac{p}{2}}$. It follows that $(\phi_\omega, \psi_\omega) \in W^{2,\frac{p}{2}} \times W^{2,\frac{p}{2}}$. By Sobolev embedding, we see that 
		\begin{align} \label{dec-pro-pro-1}
		(\phi_\omega,\psi_\omega) \in L^q \times L^q \text{ for some } q \geq \frac{p}{2} \text{ satisfying } \frac{1}{q} \geq \frac{2}{p} - \frac{2}{5}. 
		\end{align}
		We claim that $(\phi_\omega,\psi_\omega) \in L^p \times L^p$ for any $2 \leq p<\infty$. Since $(\phi_\omega,\psi_\omega) \in H^1 \times H^1$, the Sobolev embedding implies that $(\phi_\omega, \psi_\omega) \in L^p \times L^p$ for any $2 \leq p<\frac{10}{3}$. It remains to show the claim for any $p$ sufficiently large. To see it, we define the sequence
		\[
		\frac{1}{q_n} = 2^n \left( -\frac{1}{15} + \frac{2}{5 \times 2^n} \right).
		\]
		We have
		\[
		\frac{1}{q_{n+1}} -\frac{1}{q_n} = -\frac{1}{15} \times 2^n <0.
		\]
		This implies that $\frac{1}{q_n}$ is decreasing and $\frac{1}{q_n} \rightarrow -\infty$ as $n\rightarrow \infty$. Since $q_0= 3$ (we take $(\phi_\omega, \psi_\omega) \in L^3 \times L^3$ to prove our claim), it follows that there exists $k \geq 0$ such that
		\[
		\frac{1}{q_n} >0 \text{ for } 0 \leq n \leq k \text{ and } \frac{1}{q_{n+1}} \leq 0.
		\]
		We will show that $(\phi_\omega, \psi_\omega) \in L^{q_k} \times L^{q_k}$. If $(\phi_\omega, \psi_\omega) \in L^{q_{n_0}} \times L^{q_n}$ for some $0 \leq n_0 \leq k-1$, then by \eqref{dec-pro-pro-1}, $(\phi_\omega,\psi_\omega) \in L^q \times L^q$ for some $q \geq \frac{q_{n_0}}{2}$ satisfying $\frac{1}{q} \geq \frac{2}{q_{n_0}} - \frac{2}{5}$. By the choice of $q_n$, it is easy to check that $\frac{2}{q_{n_0}} - \frac{2}{5} = \frac{2}{q_{n_0+1}}$. In particular, $(\phi_\omega,\psi_\omega) \in L^{q_{n_0+1}} \times L^{q_{n_0+1}}$. By induction, we prove $(\phi_\omega, \psi_\omega) \in L^{q_k} \times L^{q_k}$. Applying again \eqref{dec-pro-pro-1}, we have
		\[
		(\phi_\omega,\psi_\omega) \in L^q \times L^q \text{ for all } q \geq \frac{q_k}{2} \text{ such that } \frac{1}{q} \geq \frac{1}{q_{k+1}}. 
		\]
		This shows that $(\phi_\omega, \psi_\omega)$ belongs to $L^p \times L^p$ for any $p$ sufficiently large. The claim follows. Using the claim, we have in particular $\psi_\omega \overline{\phi}_\omega, \phi^2_\omega \in L^p$ for any $2 \leq p<\infty$. Hence $(\phi_\omega, \psi_\omega) \in W^{2,p} \times W^{2,p}$ for any $2 \leq p<\infty$. By H\"older's inequality, we see that $\partial_j(\psi_\omega \overline{\phi}_\omega), \partial_j(\phi^2_\omega) \in L^p$ for any $2 \leq p<\infty$ and any $ 1\leq j \leq 5$. Thus $(\partial_j \phi_\omega, \partial_j \psi_\omega) \in W^{2,p} \times W^{2,p}$ for any $2 \leq p<\infty$ and any $1 \leq j \leq 5$, or $(\phi_\omega,\psi_\omega) \in W^{3,p} \times W^{3,p}$ for any $2 \leq p<\infty$. By Sobolev embedding, $(\phi_\omega,\psi_\omega) \in C^{2,\delta} \times C^{2,\delta}$ for all $0<\delta <1$. In particular, $|D^\beta \phi_\omega(x)| + |D^\beta \psi_\omega(x)| \rightarrow 0$ as $|x| \rightarrow \infty$ for all $|\beta| \leq 2$.
		
		To see the second item. Let $\epsilon>0$ and set $\chi_\epsilon(x) := e^{\frac{|x|}{1+\epsilon |x|}}$. For each $\epsilon>0$, the function $\chi_\epsilon$ is bounded, Lipschitz continuous and satisfies $|\nabla \chi_\epsilon| \leq \chi_\epsilon$ a.e. Multiplying both sides of the first equation in \eqref{ell-equ} by $\chi_\epsilon \overline{\phi}_\omega$, integrating over $\mathbb{R}^5$ and taking the real part, we have
		\[
		\text{Re} \int \nabla \phi_\omega \cdot \nabla (\chi_\omega \overline{\phi}_\omega) dx + \int \chi_\epsilon |\phi_\omega|^2 dx = 2 \text{Re} \int \chi_\epsilon \psi_\omega \overline{\phi}^2_\omega dx.
		\]
		Since $\nabla(\chi_\epsilon \overline{\phi}_\omega) = \chi_\epsilon \nabla \overline{\phi}_\omega + \nabla \chi_\epsilon \overline{\phi}_\omega$, the Cauchy-Schwarz inequality implies that
		\begin{align*}
		\text{Re} \int \nabla \phi_\omega \cdot \nabla (\chi_\epsilon \overline{\phi}_\omega) dx &= \int \chi_\epsilon |\nabla \phi_\omega|^2 dx + \text{Re} \int \nabla \chi_\epsilon \nabla \phi_\omega \overline{\phi}_\omega dx \\
		&\geq \int \chi_\epsilon |\nabla \phi_\omega|^2 dx - \int |\nabla \chi_\epsilon| |\nabla \phi_\omega| |\phi_\omega| dx \\
		& \geq \int \chi_\epsilon |\nabla \phi_\omega|^2 dx - \frac{1}{2} \int \chi_\epsilon (|\nabla \phi_\omega|^2 + |\phi_\omega|^2) dx.
		\end{align*}
		We thus get
		\begin{align} \label{dec-pro-pro-2}
		\int \chi_\epsilon (|\nabla \phi_\omega|^2 + |\phi_\omega|^2) dx \leq 4 \text{Re} \int \chi_\epsilon \psi_\omega \overline{\phi}^2_\omega dx.
		\end{align}
		Similarly, multiplying both sides of the second equation in \eqref{ell-equ} with $\chi_\epsilon \overline{\psi}_\omega$, integrating over $\mathbb{R}^5$ and taking the real part, we get
		\begin{align} \label{dec-pro-pro-3}
		\int \chi_\epsilon(|\nabla \psi_\omega|^2 + 4 |\psi_\omega|^2) dx \leq \frac{8}{3} \text{Re} \int \chi_\epsilon \overline{\psi}_\omega \phi^2_\omega dx.
		\end{align}
		By the first item, there exists $R>0$ large enough such that $|v(x)| \leq \frac{1}{8}$ for $|x| \geq R$. We have that
		\begin{align*}
		4 \text{Re} \int \chi_\epsilon \psi_\omega \overline{\phi}^2_\omega dx & \leq 4 \int \chi_\epsilon |\psi_\omega| |\phi_\omega|^2 dx \\
		&=4 \int_{|x| \leq R} \chi_\epsilon |\psi_\omega| |\phi_\omega|^2 dx + \int_{|x| \geq R} \chi_\epsilon |\psi_\omega| |\phi_\omega|^2 dx \\
		&\leq 4 \int_{|x| \leq R} e^{|x|} |\psi_\omega| |\phi_\omega|^2 dx + \frac{1}{2} \int \chi_\epsilon |\phi_\omega|^2 dx.
		\end{align*}
		We thus get from \eqref{dec-pro-pro-2} that
		\begin{align} \label{dec-pro-pro-4}
		\int \chi_\epsilon (|\nabla \phi_\omega|^2 + |\phi_\omega|^2) dx \leq 8 \int_{|x| \leq R} e^{|x|} |\psi_\omega||\phi_\omega|^2 dx.
		\end{align}
		Letting $\epsilon \rightarrow 0$, we obtain
		\[
		\int e^{|x|}( |\nabla \phi_\omega|^2 + |\phi_\omega|^2) dx \leq 8 \int_{|x| \leq R} e^{|x|} |\psi_\omega||\phi_\omega|^2 dx <\infty.
		\]
		Similarly, by \eqref{dec-pro-pro-3} and \eqref{dec-pro-pro-4},
		\begin{align*}
		\int \chi_\epsilon (|\nabla \psi_\omega|^2 + 4 |\psi_\omega|^2) dx &\leq \frac{2}{3} \left(4 \int_{|x| \leq R} e^{|x|} |\psi_\omega| |\phi_\omega|^2 dx + \frac{1}{2} \int \chi_\epsilon |\phi_\omega|^2 dx \right) \\
		&\leq \frac{16}{3} \int_{|x| \leq R} e^{|x|} |\psi_\omega| |\phi_\omega|^2 dx.
		\end{align*}
		Letting $\epsilon \rightarrow 0$, we get
		\[
		\int e^{|x|}( |\nabla \psi_\omega|^2 + 4|\psi_\omega|^2) dx \leq \frac{16}{3} \int_{|x| \leq R} e^{|x|} |\psi_\omega||\phi_\omega|^2 dx <\infty.
		\]
		The proof is complete.
	\end{proof}
	
	We also need the following virial identity related to \eqref{mas-res-Syst}. 
	\begin{lemma} \label{lem-vir-ide-ins}
		Let $d=5$ and $\kappa=\frac{1}{2}$. Let $(u_0,v_0) \in H^1 \times H^1$ be such that $(|x|u_0, |x| v_0) \in L^2 \times L^2$. Then the corresponding solution to \eqref{mas-res-Syst} with initial data $(u(0),v(0)) = (u_0,v_0)$ satisfies
		\begin{align*} 
		\frac{d^2}{dt^2} (\|xu(t)\|^2_{L^2} + 2 \|xv(t)\|^2_{L^2}) = 8 \left(\|\nabla u(t)\|^2_{L^2} + \frac{1}{2} \|\nabla v(t)\|^2_{L^2}\right) - 20 \emph{Re} \int \overline{v}(t) u^2(t) dx.
		\end{align*}
	\end{lemma}
	\begin{proof}
		The above identity follows immediately from \cite[Lemma 3.1]{Dinh} with $\chi(x) = |x|^2$. 
	\end{proof}
		
	Now let us denote for $(u,v) \in H^1 \times H^1 \backslash \{(0,0)\}$,
	\[
	Q(u,v) := K(u,v) - \frac{5}{2} P(u,v).
	\]
	It is obvious that
	\begin{align} \label{vir-ide-ins}
	\frac{d^2}{dt^2} (\|xu(t)\|^2_{L^2} + 2 \|xv(t)\|^2_{L^2}) = 8 Q(u(t),v(t)).
	\end{align}
	Note that if we take
	\begin{align} \label{scaling}
	u^\gamma(x) = \gamma^{\frac{5}{2}} u (\gamma x), \quad v^\gamma(x) = \gamma^{\frac{5}{2}} v(\gamma x),
	\end{align}
	then 
	\begin{align*}
	S_\omega(u^\gamma, v^\gamma) &= \frac{1}{2} K(u^\gamma,v^\gamma) + \frac{\omega}{2} M(u^\gamma,v^\gamma) - P(u^\gamma,v^\gamma) \\
	&=\frac{\gamma^2}{2} K(u,v) + \frac{\omega}{2} M(u,v) - \gamma^{\frac{5}{2}} P(u,v).
	\end{align*}
	It is easy to see that
	\[
	Q(u,v) = \left. \partial_\gamma S_\omega(u^\gamma, v^\gamma) \right|_{\gamma=1}.
	\]
	
	\begin{lemma} \label{lem-cha-gro-sta-5D}
		Let $d=5$, $\kappa=\frac{1}{2}$ and $\omega>0$. Let $(\phi_\omega,\psi_\omega) \in \mathcal{G}_\omega$. Then
		\[
		S_\omega(\phi_\omega,\psi_\omega) = \inf \left\{ S_\omega(u,v) \ : \ (u,v) \in H^1 \times H^1 \backslash \{(0,0)\}, Q(u,v)=0 \right\}.
		\]
	\end{lemma}
	
	\begin{proof}
		Denote $m:= \inf \left\{ S_\omega(u,v) \ : \ (u,v) \in H^1 \times H^1 \backslash \{(0,0)\}, Q(u,v)=0 \right\}$. Since $(\phi_\omega,\psi_\omega)$ is a solution of \eqref{ell-equ}, it follows from Lemma $\ref{lem-poh-ide}$ that $Q(\phi_\omega,\psi_\omega) =K_\omega(\phi_\omega,\psi_\omega)=0$. Thus
		\begin{align} \label{cha-gro-sta-5D-pro-1}
		S_\omega(\phi_\omega,\psi_\omega) \geq m.
		\end{align}
		Now let $(u,v) \in H^1 \times H^1 \backslash \{(0,0)\}$ be such that $Q(u,v) =0$. If $K_\omega(u,v) =0$, then by Proposition $\ref{prop-exi-gro-sta-ins}$, $S_\omega(u,v) \geq S_\omega(\phi_\omega,\psi_\omega)$. If $K_\omega(u,v) \ne 0$, we consider $K_\omega(u^\gamma, v^\gamma) = \gamma^2 K(u,v) + \omega M(u,v) - \gamma^{\frac{5}{2}} P(u,v)$, where $(u^\gamma, v^\gamma)$ is as in \eqref{scaling}. Since $\lim_{\gamma \rightarrow 0} K_\omega(u^\gamma, v^\gamma)= \omega M(u,v) >0$ and $\lim_{\gamma \rightarrow \infty} K_\omega(u^\gamma, v^\gamma) = -\infty$, there exists $\gamma_0>0$ such that $K_\omega(u^{\gamma_0},v^{\gamma_0}) =0$. It again follows from Proposition $\ref{prop-exi-gro-sta-ins}$, $S_\omega(u^{\gamma_0},v^{\gamma_0}) \geq S_\omega(\phi_\omega,\psi_\omega)$. On the other hand,
		\[
		\partial_\gamma S_\omega(u^\gamma,v^\gamma) = \gamma K(u,v) - \frac{5}{2} \gamma^{\frac{3}{2}} P(u,v).
		\]
		We see that the equation $\partial_\gamma S_\omega(u^\gamma, v^\gamma) =0$ admits a unique non-zero solution 
		\[
		\left( \frac{2K(u,v)}{5P(u,v)} \right)^2=1
		\]
		since $Q(u,v) =0$. This implies that $\partial_\gamma S_\omega(u^\gamma, v^\gamma)>0$ if $\gamma \in (0,1)$ and $\partial_\gamma S_\omega(u^\gamma,v^\gamma)<0$ if $\gamma \in (1,\infty)$. In particular, $S_\omega(u^\gamma,v^\gamma) \leq S_\omega(u,v)$ for all $\gamma >0$. Hence $S_\omega(u^{\gamma_0},v^{\gamma_0}) \leq S_\omega(u,v)$. We thus obtain $S_\omega(\phi_\omega,\psi_\omega) \leq S_\omega(u,v)$ for any $(u,v) \in H^1 \times H^1 \backslash \{(0,0)\}$ satisfying $Q(u,v)=0$. Therefore,
		\begin{align} \label{cha-gro-sta-5D-pro-2}
		S_\omega(\phi_\omega,\psi_\omega) \leq m.
		\end{align}
		Combining \eqref{cha-gro-sta-5D-pro-1} and \eqref{cha-gro-sta-5D-pro-2}, we prove the result.
	\end{proof}
	
	Let $(\phi_\omega,\psi_\omega) \in \mathcal{G}_\omega$. Define
	\[
	\mathcal{B}_\omega:= \left\{ (u,v) \in H^1 \times H^1 \backslash \{(0,0)\} \ : \ S_\omega(u,v) < S_\omega(\phi_\omega,\psi_\omega), Q(u,v) <0 \right\}.
	\]
	
	\begin{lemma} \label{lem-inv-set}
		Let $d=5$, $\kappa=\frac{1}{2}$, $\omega>0$ and $(\phi_\omega,\psi_\omega) \in \mathcal{G}_\omega$. The set $\mathcal{B}_\omega$ is invariant under the flow of \eqref{mas-res-Syst}.
	\end{lemma}
	
	\begin{proof}
		Let $(u_0,v_0) \in \mathcal{B}_\omega$. We will show that the corresponding solution $(u(t),v(t))$ to \eqref{mas-res-Syst} with initial data $(u(0),v(0)) = (u_0,v_0)$ satisfies $(u(t),v(t)) \in \mathcal{B}_\omega$ for any $t$ in the existence time.  Indeed, by the conservation of mass and energy, we have
		\begin{align} \label{inv-set-pro}
		S_\omega(u(t),v(t)) = S_\omega(u_0,v_0) < S_\omega (\phi_\omega,\psi_\omega)
		\end{align}
		for any $t$ in the existence time. It remains to show that $Q(u(t),v(t))<0$ for any $t$ as long as the solution exists. Suppose that there exists $t_0 >0$ such that $Q(u(t_0),v(t_0)) \geq 0$. By the continuity of the function $t\mapsto Q(u(t),v(t))$, there exists $t_1 \in (0,t_0]$ such that $Q(u(t_1),v(t_1)) =0$. It follows from Lemma $\ref{lem-cha-gro-sta-5D}$ that $S_\omega(u(t_1),v(t_1)) \geq S_\omega(\phi_\omega,\psi_\omega)$ which contradicts to \eqref{inv-set-pro}. The proof is complete.
	\end{proof}
	
	\begin{lemma} \label{lem-key-lem}
		Let $d=5$, $\kappa=\frac{1}{2}$, $\omega>0$ and $(\phi_\omega,\psi_\omega) \in \mathcal{G}_\omega$. If $(u,v) \in \mathcal{B}_\omega$, then
		\[
		Q(u,v) \leq 2 (S_\omega(u,v) - S_\omega(\phi_\omega,\psi_\omega)).
		\]
	\end{lemma}
	
	\begin{proof}
		Let $(u,v) \in \mathcal{B}_\omega$. Set 
		\[
		f(\gamma):= S_\omega(u^\gamma,v^\gamma) = \frac{\gamma^2}{2} K(u,v) + \frac{\omega}{2} M(u,v) - \gamma^{\frac{5}{2}} P(u,v).
		\]
		We have
		\[
		f'(\gamma) = \gamma K(u,v) - \frac{5}{2} \gamma^{\frac{3}{2}} P(u,v) = \frac{Q(u^\gamma, v^\gamma)}{\gamma}.
		\]
		We see that
		\begin{align} \label{key-lem-pro}
		(\gamma f'(\gamma))' &= 2\gamma K(u,v) - \frac{25}{4} \gamma^{\frac{3}{2}} P(u,v) \nonumber \\
		&= 2 \left(\gamma K(u,v) - \frac{5}{2} \gamma^{\frac{3}{2}} P(u,v) \right) - \frac{5}{4} \gamma^{\frac{3}{2}} P(u,v) \nonumber \\
		&\leq 2f'(\gamma)
		\end{align}
		for all $\gamma >0$. Note that $P(u,v) \geq 0$ which follows from the fact $Q(u,v) <0$. We also note that since $Q(u,v) <0$, the equation $\partial_\gamma S_\omega(u^\gamma, v^\gamma)=0$ admits a unique non-zero solution 
		\[
		\gamma_0 = \left(\frac{2K(u,v)}{5P(u,v)} \right)^2 \in (0,1),
		\]
		and $Q(u^{\gamma_0},v^{\gamma_0}) = \gamma_0 \times \left.\partial_\gamma S_\omega(u^\gamma,v^\gamma)\right|_{\gamma=\gamma_0} =0$. Taking integration of \eqref{key-lem-pro} over $(\gamma_0,1)$ and using the fact $\gamma f'(\gamma) = Q(u^\gamma,v^\gamma)$, we get
		\[
		Q(u,v) - Q(u^{\gamma_0},v^{\gamma_0}) \leq 2 (S_\omega(u,v) - S_\omega(u^{\gamma_0},v^{\gamma_0})).
		\]
		The result then follows from the fact that $S_\omega(\phi_\omega,\psi_\omega) \leq S_\omega(u^{\gamma_0},v^{\gamma_0})$ since $Q(u^{\gamma_0},v^{\gamma_0}) = 0$.
	\end{proof}
	
	We are now able to prove the strong instability of standing waves given in Theorem $\ref{theo-str-ins}$. 
	
	\noindent \textit{Proof of Theorem $\ref{theo-str-ins}$.} 
	Let $\epsilon>0$. Since $(\phi^{\gamma}_\omega, \psi^{\gamma}_\omega) \rightarrow (\phi_\omega,\psi_\omega)$ as $\gamma \rightarrow 1$, there exists $\gamma_0>1$ such that $\|(\phi^{\gamma_0}_\omega,\psi^{\gamma_0}_\omega) - (\phi_\omega,\psi_\omega)\|_{H^1 \times H^1} <\epsilon$. We claim that $(\phi^{\gamma_0}_\omega, \psi^{\gamma_0}_\omega) \in \mathcal{B}_\omega$. Indeed, we have
	\begin{align*}
	S_\omega(\phi^\gamma_\omega, \psi^\gamma_\omega) &= \frac{\gamma^2}{2} K(\phi_\omega,\psi_\omega) +\frac{\omega}{2} M(\phi_\omega,\psi_\omega) -\gamma^{\frac{5}{2}} P(\phi_\omega,\psi_\omega), \\
	\partial_\gamma S_\omega(\phi^\gamma_\omega, \psi^\gamma_\omega) &= \gamma K(\phi_\omega, \psi_\omega) - \frac{5}{2} \gamma^{\frac{3}{2}} P(\phi_\omega,\psi_\omega) = \frac{Q(\phi^\gamma_\omega, \psi^\gamma_\omega)}{\gamma}.
	\end{align*}
	Since $Q(\phi_\omega,\psi_\omega)=0$, the equation $\partial_\gamma S_\omega(\phi^\gamma_\omega, \psi^\gamma_\omega) =0$ admits a unique non-zero solution 
	\[
	\left( \frac{2K(\phi_\omega,\psi_\omega)}{5P(\phi_\omega,\psi_\omega)}  \right)^2 =1.
	\]
	This implies that $\partial_\gamma S_\omega(\phi^\gamma_\omega, \psi^\gamma_\omega) >0$ if $\gamma \in (0,1)$ and $\partial_\gamma S_\omega(\phi^\gamma_\omega, \psi^\gamma_\omega)<0$ if $\gamma \in (1,\infty)$. In particular, $S_\omega(\phi^\gamma_\omega,\psi^\gamma_\omega)<S_\omega(\phi_\omega,\psi_\omega)$ for any $\gamma>0$ and $\gamma \ne 1$. On the other hand, since $Q(\phi^\gamma_\omega,\psi^\gamma_\omega)= \gamma \partial_\gamma S_\omega(\phi^\gamma_\omega, \psi^\gamma_\omega)$, we see that $Q(\phi^\gamma_\omega, \psi^\gamma_\omega) >0$ if $\gamma \in (0,1)$ and $Q(\phi^\gamma_\omega, \psi^\gamma_\omega)<0$ if $\gamma \in (1,\infty)$. Since $\gamma_0>1$, we see that
	\[
	S_\omega(\phi^{\gamma_0}_\omega, \psi^{\gamma_0}_\omega)< S_\omega(\phi_\omega,\psi_\omega) \text{ and } Q(\phi^{\gamma_0}_\omega,\psi^{\gamma_0}_\omega) <0.
	\]
	Therefore, $(\phi^{\gamma_0}_\omega, \psi^{\gamma_0}_\omega) \in \mathcal{B}_\omega$ and the claim follows.
	
	By the local well-posedness, there exists a unique solution $(u(t), v(t)) \in C([0,T), H^1 \times H^1)$ to \eqref{mas-res-Syst} with initial data $(u(0),v(0)) = (\phi^{\gamma_0}_\omega, \psi^{\gamma_0}_\omega)$, where $T>0$ is the maximal time of existence. By Lemma $\ref{lem-inv-set}$, we see that $(u(t),v(t)) \in \mathcal{B}_\omega$ for any $t\in [0,T)$. Thus, applying Lemma $\ref{lem-key-lem}$, we get 
	\[
	Q(u(t),v(t)) \leq 2 (S_\omega(u(t),v(t)) - S_\omega (\phi_\omega,\psi_\omega)) = 2(S_\omega(\phi^{\gamma_0},\psi^{\gamma_0}) - S_\omega(\phi_\omega, \psi_\omega)) =- \delta
	\]
	for any $t\in [0,T)$, where $\delta= 2 (S_\omega(\phi_\omega, \psi_\omega) - S_\omega(\phi^{\gamma_0}_\omega, \psi^{\gamma_0}_\omega)) >0$. Since $(|x|\phi_\omega, |x|\psi_\omega) \in L^2 \times L^2$, it follows that $(|x| \phi^{\gamma_0}_\omega, |x| \psi^{\gamma_0}_\omega) \in L^2 \times L^2$. Thanks to the virial identity \eqref{vir-ide-ins}, we obtain
	\[
	\frac{d^2}{dt^2} \left( \|xu(t)\|^2_{L^2} + 2 \|xv(t)\|^2_{L^2} \right) = 8 Q(u(t),v(t)) \leq -8 \delta <0,
	\]
	for any $t\in [0,T)$. The classical argument of Glassey \cite{Glassey} implies that the solution blows up in finite time. The proof is complete.
	
	\hfill $\Box$

	\section*{Acknowledgement}
	The author would like to express his deep gratitude to his wife - Uyen Cong for her encouragement and support. He also would like to thank the reviewer for his/her helpful comments and suggestions.

\end{document}